\documentclass[12pt]{amsproc}

\usepackage[english]{babel}
\usepackage[utf8x]{inputenc}
\usepackage{euscript}
\usepackage{graphicx} 
\usepackage{epstopdf}
\usepackage{amscd}
\usepackage{array} 
\usepackage{verbatim} 
\usepackage{amssymb}
\usepackage{amsmath}
\usepackage{amsthm}
\usepackage[colorlinks,linktocpage, colorlinks=true,
            linkcolor=red,
            urlcolor=blue,
            citecolor=blue]{hyperref}
\usepackage{backref}
\usepackage{amsrefs, todonotes}

\usepackage[
  a4paper,
  left=29mm,
  right=29mm,
  top=29mm,
  bottom=31mm
]{geometry}

\newtheorem{theorem}{Theorem}[section]
\newtheorem{lemma}[theorem]{Lemma}
\newtheorem{proposition}[theorem]{Proposition}

\newtheorem{claim}[theorem]{Claim}

\theoremstyle{definition}
\newtheorem{definition}[theorem]{Definition}
\theoremstyle{remark}
\newtheorem{remark}[theorem]{Remark}
\theoremstyle{remark}

\theoremstyle{remark}

\theoremstyle{remark}

\renewcommand{\phi}{\varphi}

\begin{document}

\title[Inapproximability of actions]{Inapproximability of actions and \\ Kazhdan's property (T)}

\author{G\'abor Kun}
\address{G\'abor Kun, HUN-REN Alfréd Rényi Institute of Mathematics, Budapest, Hungary\\
and Eötvös Loránd University, Budapest, Hungary}
\email{kun.gabor@renyi.mta.hu}

\author{Andreas Thom}
\address{Andreas Thom, TU Dresden, Germany}
\email{andreas.thom@tu-dresden.de}

\begin{abstract}
Let $\Gamma$ be a countable group with Kazhdan's property~$(T)$ and let
$\Delta$ be a finitely generated group. We prove that if a sofic
approximation of $\Gamma \times \Delta$ has the property that its spanning
subgraphs consisting of the $\Gamma$-labeled edges form an expander
sequence, then $\Delta$ is locally embeddable into finite groups (LEF).
Consequently, if $\Delta$ is not LEF, then every almost free p.m.p.\ action
of $\Gamma \times \Delta$ whose restriction to $\Gamma$ is ergodic fails
to weakly contain any sequence of finite labeled graphs; in particular,
such an action is not a local-global limit of finite graphs. We also prove
that, when $\Delta$ is amenable, the existence of an expander sofic
approximation of $\Gamma \times \Delta$ forces $\Delta$ to be LEF. The
main technical input is a self-improvement theorem for almost
automorphisms of expander sofic approximations of property~$(T)$ groups.
It implies that the Hamming-distance clusters of sufficiently accurate
almost automorphisms carry a natural finite group structure.
\end{abstract}


\maketitle

\tableofcontents

\section{Introduction}

A sequence of finite labeled graphs is locally convergent if, for every
$r \in \mathbb N$, the isomorphism class of the rooted $r$-ball centered
at a uniformly chosen vertex converges in distribution. Local-global
convergence refines local convergence by recording the corresponding
colored neighborhood statistics for every finite vertex-coloring
\cites{bollobas,hatami}. A probability measure preserving action of a
finitely generated group, together with a finite generating set, determines
a labeled graphing. It is natural to ask whether such a graphing can be
obtained as a local-global limit of finite labeled graphs.

A finitely generated group is called \emph{sofic} if its labeled Cayley
graph admits a local approximation by finite labeled graphs. Sofic groups
were introduced by Gromov \cite{gromov}; see also Weiss \cite{weiss} and
the surveys \cites{pestov,icm}. Although the local structure of a sofic
approximation is prescribed by the Cayley graph, its global structure can
vary considerably. Schramm proved that sofic approximations of amenable
groups are hyperfinite \cite{schramm}. At the opposite end of the spectrum,
the first author proved that every sofic approximation of a Kazhdan group
is essentially a vertex-disjoint union of expander graphs
\cite{kunkazhdan}.

We build on this rigidity result in order to study sofic approximations of
direct products $\Gamma \times \Delta$. The main observation is that the
permutations defined by the $\Delta$-labels almost commute with the
$\Gamma$-labels and hence give almost automorphisms of the
$\Gamma$-labeled subgraphs. Expansion and Kazhdan's property~$(T)$ impose
strong restrictions on these almost automorphisms.

Recall that a group $\Lambda$ is {\rm LEF}, or locally embeddable into
finite groups, if, for every finite subset $F \subseteq \Lambda$, there
exist a finite group $H$ and an injective map
$\varphi \colon F \rightarrow H$ such that
$
\varphi(x)\varphi(y)=\varphi(xy)
$
whenever $x,y,xy \in F$. Every residually finite group is {\rm LEF}, and a
finitely presented group is {\rm LEF} if and only if it is residually
finite. Our first main result is the following.

\begin{theorem}~\label{main}
Let $\Gamma$ be a countable Kazhdan group and let $\Delta$ be a finitely
generated group. Let $S_{\Gamma}$ and $S_{\Delta}$ be finite generating
sets of $\Gamma$ and $\Delta$, respectively. Consider a sofic approximation
of $\Gamma \times \Delta$ with respect to the generating set
$S_{\Gamma} \cup S_{\Delta}$. If the spanning subgraphs consisting of the
edges labeled by elements of $S_{\Gamma}$ form an expander sequence, then
$\Delta$ is {\rm LEF}. In particular, if $\Delta$ is finitely presented,
then it is residually finite.
\end{theorem}

\begin{remark}
By \cite{kunkazhdan}, for every sofic approximation of a Kazhdan group
$\Gamma$, the edges labeled by elements of $S_{\Gamma}$ induce,
essentially, a vertex-disjoint union of expander graphs. Theorem~\ref{main}
requires the stronger assumption that the resulting spanning subgraphs
form, essentially, a single expander sequence.
\end{remark}

The mechanism behind Theorem~\ref{main} is as follows. Since the two
factors commute in $\Gamma \times \Delta$, the permutations defined by
elements of $\Delta$ are almost automorphisms of the $\Gamma$-labeled
subgraphs of a sufficiently good sofic approximation. The composition of
two $\varepsilon$-almost automorphisms need only be a
$2\varepsilon$-almost automorphism, so the almost automorphisms at a fixed
error scale do not themselves form a group. We prove, however, that they
can be improved without moving them far in Hamming distance. Expansion
then implies that sufficiently accurate almost automorphisms split into
well-separated Hamming-distance clusters, and these clusters carry a
natural finite group structure. Given a finite subset of $\Delta$, the
corresponding elements determine distinct clusters, and multiplication is
preserved whenever the relevant products remain in the finite subset.
This produces the required partial embedding into a finite group.

As a consequence of Theorem~\ref{main}, we obtain p.m.p.\ actions which
cannot be approximated by finite labeled graphs in the local-global sense
of Hatami, Lov\'asz and Szegedy \cite{hatami}, based on the colored
neighborhood statistics of Bollob\'as and Riordan \cite{bollobas}. We use
weak containment in the one-sided local-global sense described in
Section~2, following the framework of \cite{kechris}.

\begin{theorem}\label{main2}
Let $\Gamma$ be a countable Kazhdan group and let $\Delta$ be a finitely
generated group which is not {\rm LEF}. Consider an almost free,
probability measure preserving action of $\Gamma \times \Delta$ on a
probability measure space such that the restriction of the action to
$\Gamma$ is ergodic. Then this action does not weakly contain any sequence
of finite labeled graphs. In particular, it is not a local-global limit of
finite graphs.
\end{theorem}

\begin{proof}[Proof of Theorem~\ref{main2}]
Suppose, toward a contradiction, that a sequence
 of finite
$S_{\Gamma} \cup S_{\Delta}$-labeled graphs $\{G_n\}_{n=1}^{\infty}$ is weakly contained in the
action. Taking the one-coloring in the definition of weak containment
shows that the uncolored rooted neighborhood statistics of $G_n$ converge
to those of the action. Since the action is almost free, these are the
rooted neighborhood statistics of the Cayley graph of
$\Gamma \times \Delta$. Hence $\{G_n\}_{n=1}^{\infty}$ is a sofic
approximation of $\Gamma \times \Delta$.

Restrict the labeling to $S_{\Gamma}$. By the main result of
\cite{kunkazhdan}, after changing $o(|V(G_n)|)$ labeled edges, the resulting
spanning subgraphs are vertex-disjoint unions of expanders with a uniform
expansion constant.

The restriction of the given action to $\Gamma$ is strongly ergodic,
because it is ergodic and $\Gamma$ is a Kazhdan group. We claim that one
of the $\Gamma$-components must contain almost all of the vertices.  After
discarding the remaining $o(|V(G_n)|)$ vertices and changing
$o(|V(G_n)|)$ labeled edges in order to restore the regular labeling, we
obtain a sofic approximation of $\Gamma \times \Delta$ whose
$\Gamma$-labeled subgraphs form an expander sequence. Theorem~\ref{main}
then implies that $\Delta$ is {\rm LEF}, which is a contradiction.
\end{proof}

When $\Gamma$ is infinite, the Bernoulli shift of
$\Gamma \times \Delta$ over a nontrivial probability space is almost free,
and its restriction to $\Gamma$ is mixing. It therefore gives a basic
example to which Theorem~\ref{main2} applies.

Theorem~\ref{main} concerns expansion of the subgraphs defined by the
$\Gamma$-labels. When $\Delta$ is amenable, expansion of the entire sofic
approximation already forces the same conclusion.

\begin{theorem}\label{product}
Let $\Gamma$ be a countable Kazhdan group and let $\Delta$ be a finitely
generated amenable group. If $\Gamma \times \Delta$ admits a sofic
approximation by a sequence of expander graphs, then $\Delta$ is {\rm LEF}.
\end{theorem}

The proof again begins with the expander decomposition of the
$\Gamma$-labeled subgraphs and then uses hyperfiniteness of the induced action of $\Delta$ on the set of components.

We finally describe the technical input used in the proof of
Theorem~\ref{main}. An $\varepsilon$-almost automorphism of a finite labeled
graph is a bijection which fails to preserve at most
$\varepsilon |V(G)|$ labeled edges. The graph of such a bijection is a
small-boundary subset of the product graph $G \times G$. The property~$(T)$
estimate from \cite{kunkazhdan} provides a nearby subset with arbitrarily
small boundary. 
Expansion also implies a Hamming-distance gap: two sufficiently accurate
almost automorphisms are either close or almost maximally far apart.
Consequently, closeness defines an equivalence relation on sufficiently
accurate almost automorphisms. Composition, followed by the improvement
procedure, defines a well-defined group operation on the equivalence
classes. The resulting finite group of clusters is constructed in
Section~4 and is the finite group used in the proof of
Theorem~\ref{main}.

The use of expansion and Kazhdan's property~$(T)$ to control almost
symmetries of finite permutation models is closely related to work of
Alekseev and the second author. In
\cite{AlekseevThomUniversalSofic}, related control of centralizers in
universal sofic groups is used to prove that there are
$2^{\aleph_0}$ non-isomorphic universal sofic groups, confirming a
conjecture of Thomas \cite{thomas}. Related permanence properties of
approximation and stability under quotients are studied in
\cite{AlekseevThomApproxStability}.

\medskip

This paper is organized as follows. In Section~2 we recall the basic
notions used throughout the paper. Section~3 contains the main technical
input, namely the improvement theorem for almost automorphisms of expander
sofic approximations of Kazhdan groups. In Section~4 we construct the
finite group of clusters of almost automorphisms and prove
Theorem~\ref{main}. Finally, Section~5 proves Theorem~\ref{product}.

\medskip

This paper was mostly written in 2018 and a first version appeared without Section \ref{sec:product} on the arXiv in January 2019, but was never submitted for publication. The authors spoke about the results of Section \ref{sec:product} on various occasions over the last years. We finally finished writing it by mid July 2026. On August 1, 2026, we learned about the breakthrough result of OpenAI. The proof of their key Proposition 2.3 in \cite{OpenAI26} is a creative and effective application of the results and techniques of this paper and the ones in \cite{kunkazhdan}.

\section{Definitions}

Throughout, a graph is denoted by $G=(V(G),E(G))$, where $V(G)$ is the set of vertices and $E(G)$ the set of undirected edges. 
We work with sequences $\{G_n\}_{n=1}^{\infty}$ of finite, undirected, $d$-regular graphs. The edge boundary of $S \subseteq V(G)$ is
defined to be $\partial(S)=E(S,V(G) \setminus S)$. The Cheeger constant of the graph $G$ is 
$$h(G) = \min_{S \subseteq V(G), |S| \leq |V(G)|/2} |\partial S|/|S|.$$ 
We say that a sequence $\{G_n\}_{n=1}^{\infty}$ of finite, undirected, $d$-regular graphs is an expander sequence if there exists an $h>0$ such that $h(G_n) \geq h$
holds for $n \in \mathbb N$. 

\vspace{0.1cm}

Let $\Gamma$ denote a group generated by the finite symmetric set $S=S^{-1} \subset \Gamma$. We will consider $|S|$-regular graphs 
with an edge-labeling, where every edge will be labeled by an element of $S$ in a directed way, and the label of $(x,y)$ will be the inverse of the
label of $(y,x)$ for any pair of vertices. We say that an $|S|$-regular graph $G$ is regularly $S$-labeled if for every vertex $x$ and $s \in S$ there 
exists a unique vertex $y$ such that $(x,y)$ is an edge labeled by $s$, and we write $y=sx$. We say that a sequence $\{G_n\}_{n=1}^{\infty}$ is a 
sofic approximation of $\Gamma$ if it is a sequence of regularly $S$-labeled graphs and for every $s_1, \dots ,s_k \in S$ and $\varepsilon>0$ if $s_1 \dots s_k = 1$
then $x=s_1 \dots s_k x$ holds for all but at most an $\varepsilon$-proportion of the vertices if $n$ is large enough, while if $s_1 \dots s_k \neq 1$ then $x \neq s_1 \dots s_k x$ holds for all but 
at most an $\varepsilon$-proportion of the vertices if $n$ is large enough.

\vspace{0.1cm}

We say that a group is residually finite if it can be embedded into the direct product of finite groups. Residually finite groups are sofic. We say that a group $\Gamma$ is LEF (Locally Embeddable into Finite groups)  if for any finite subset $F \subseteq \Gamma$ there exists a finite group $H$ and an {injective }mapping $\varphi: F \rightarrow H$ such that $\varphi(x) \varphi(y) = \varphi(xy)$ holds if $x,y, xy \in F$. It is easy to see that a group is LEF if it is a subgroup of an ultraproduct of finite groups. Every finitely presented LEF group is residually finite.

\vspace{0.1cm}

We will study probability measure preserving (p.m.p.\ for short) almost free actions of groups on probability measure spaces. 
A p.m.p.\  action of a group $\Gamma$ generated by a finite set $S$ on a probability measure space $X$ is ergodic if for any measurable $A \subseteq X, \nu(A) \neq 0,1$ 
there exists $s \in S$ such that $\nu(sA \setminus A)>0$. It is strongly ergodic if there exists an $h>0$ such that for any measurable $A \subseteq X, \nu(A) \leq 1/2$ there exists $s \in S$ such that $\nu(sA \setminus A)>h\nu(A)$. By a result of Connes and Weiss, a group has Kazhdan's Property (T) if every ergodic action of the group is strongly ergodic.
However, we will use a different definition of Kazhdan's property (T) using the terminology of \cite{kunkazhdan}. We say that the finitely generated group $\Gamma$ 
has Kazhdan Property (T) if there is a finite set of generators $S$ and a Kazhdan constant $\kappa>0$, such that for every Hilbert space $\mathcal{H}$ and 
$\pi\colon \Gamma \rightarrow U(\mathcal{H})$ a unitary representation of $\Gamma$, either $\pi$ has a non-zero invariant vector or for $A = \sum_{s \in S} \pi(s)/|S|$ the inequality 
$\|A\xi \| \leq (1-\kappa) \| \xi \|$ holds for every $\xi \in \mathcal{H}$. See the book of Bekka, de la Harpe and Valette  \cite{bachir} for more on Property (T).

 We will define the notion of weak containment in the spirit of Kechris \cite{kechris} for sequences of (labeled) graphs. 
Given an integer $r$ and a graph $G$ we will consider the following probability distribution on isomorphism classes of rooted graphs:
Pick a vertex uniformly at random and consider the isomorphism class of the rooted $r$-ball centered and rooted at $x$. Given a finite set of colors we can extend this to vertex-colored graphs and we call this the colored $r$-neighborhood statistics of the graph following Bollob\'as and Riordan \cite{bollobas}. An almost free p.m.p.\  action of a group $\Gamma$ generated by a finite set $S$ on a probability measure space $X$ weakly contains a sequence of finite $S$-regularly labeled graphs $\{G_n\}_{n=1}^{\infty}$ if for every $k,r$ and $\varepsilon>0$ there exists $N$, such that if $n \geq N$ then for any $k$-coloring of $G_n$ there exists a measurable coloring of $X$ that is $\varepsilon$-close to the colored $r$-neighborhood statistics of the finite colored graph $G_n$.

\section{Improving almost automorphism}

We introduce the notion of $\varepsilon$-automorphism for some $\varepsilon>0$, this is a key notion that we will use in the sequel.

\begin{definition}
Let $(V,E)$ be a finite edge-labeled graph $(V,E)$ and $\varepsilon>0$. A {bijection} $c \colon V \rightarrow V$ is called an $\varepsilon$-almost automorphism, if there are at most
$\varepsilon|V|$ edges, whose image under $c$ is not an edge with the same label.
\end{definition}
Note that automorphisms are exactly the $0$-almost automorphisms. It is tempting to study the set of $\varepsilon$-automorphisms of a finite graph appearing in a sofic approximation of a group $\Gamma$. However, the product of two $\varepsilon$-automorphisms is in general only a $2\varepsilon$-automorphism and that is a problem that renders this idea rather useless unless there is some mechanism that allows us to \emph{improve} the resulting $2\varepsilon$-automorphisms. This is exactly what the main theorem of this section achieves under additional assumptions on $\Gamma$.




Let $G$ denote a finite, simple, regularly $S$-labeled graph and let $b \colon V(G) \rightarrow V(G)$ be a map. The graph of $b$ will be denoted by 
$B = \{ (x,b(x)) : x \in V(G) \} \subseteq V(G) \times V(G)$. Note that $G \times G$ is also a regularly $S$-labeled graph.
\begin{lemma}
Then, $b$ is an $\varepsilon$-automorphism if and only if $|\partial B| \leq 2\varepsilon |V(G)|$. \end{lemma}
\begin{proof}
For every $s$-labeled edge $(x,y) \in E(G)$ the followings are equivalent: 

\begin{enumerate}

\item
{$(b(x), b(y))$ is not an $s$-labeled edge (maybe not even an edge) of $V(G)$.}

\item
{$s(x,b(x)) \notin B$}

\item
{$s^{-1}(y,b(y)) \notin B$}

\end{enumerate}

This gives a one-to-one correspondence between the bad edges of the bijection $b$ and pairs of edges on the boundary of $\partial B$. 
Here we use that $G$ is simple, hence $x \neq y$. The lemma follows.
\end{proof}

We will use the structural results of the first author about the sofic approximation of Kazhdan groups.  The main result of \cite{kunkazhdan} is that every sofic approximation of a countable Kazhdan group is essentially a disjoint union of expander graphs. However, we will need a specific proposition that will be relevant in the study of the expansion of small sets, in particular the graph of an almost automorphism. 

\begin{proposition}[see \cite{kunkazhdan}] \label{prop} 
Let $\Gamma$ be a finitely generated Kazhdan group with finite and symmetric generating set $S$, where $1 \in S$, and Kazhdan constant $\kappa$.
For every $\alpha>0$ there exists an integer $r$ such that for any finite, $S$-edge labeled graph $G$ and $T \subseteq V(G)$, where the  ball $B(t,r)$ is isomorphic to the $r$-ball in the Cayley graph of $\Gamma$ for every $t$ in $T$, there exists a set $U$ such that $$|\partial U | \leq \alpha |U| \quad \mbox{and} \quad | U \triangle T | \leq \frac{10}{\kappa^2} \|\chi_T - M\chi_T \|^2 \leq \frac{5 |\partial T|}{d\kappa^2}.$$
\end{proposition}

We will use this proposition in order to prove our main technical result.
The Hamming distance of two permutations will be denoted by 
$$d_{\rm H}(\sigma,\tau) := |\{1 \leq i \leq n \mid \sigma(i) \neq \tau(i)\}|.$$

\begin{theorem} Let $\Gamma$ be a finitely generated Kazhdan group with a fixed finite and symmetric generating set $S$ and Kazhdan constant $\kappa$. There exists $C,\varepsilon_0>0$ depending only on $S$ and $\kappa$, such that for all $0 \leq \varepsilon\leq \varepsilon_0$ the following holds:

Let $(G_n)_{n=1}^{\infty}$ be a sofic approximation by a sequence of regularly $S$-labeled
expander graphs and for each $n \in \mathbb N$ let $c_n \colon V(G_n) \to V(G_n)$ be an $\varepsilon$-almost automorphism. Then for every $\delta>0$ and $n \in \mathbb N$ large enough, there is a $\delta$-almost automorphism $c'_n\colon V(G_n) \to V(G_n)$ such that 
$$d_{\rm H}(c_n,c'_n) \leq \varepsilon C|V(G_n)|.$$
\end{theorem}

\begin{proof}
We may assume that the graphs are regularly $S$-labeled.
We choose $\varepsilon_0$ and $C$ later. 
Consider the sofic approximation $G_n \times G_n$ of $\Gamma$ given as the product of the sofic approximations and let $F_n \subset V(G_n) \times V(G_n)$ denote the graph of $c_n$. In order to improve $F_n$ we apply Proposition \ref{prop} to $\alpha$ small enough chosen later 
in order to get an $r$ for which the conditions of the proposition hold.
Let $L_n^r \subseteq V(G_n) \times V(G_n)$ denote the set of vertices whose 
$r$-neighborhood is isomorphic to the rooted $r$-ball in the Cayley graph of $\Gamma$. 
The set $L_n^r$ is a product set: 
$L_n^r = K_n^r \times K_n^r$, where $K_n^r$ denotes the set of vertices in $V(G_n)$ whose
$r$-neighborhood is isomorphic to the rooted $r$-ball in the Cayley graph of $\Gamma$. 
The sequence  $(G_n)_{n=1}^{\infty}$ is a sofic approximation, hence $$|V(G_n) \setminus K_n^r| = o_n(|V(G_n)|).$$
Since $c_n$ is a bijection $|F_n \setminus L_n^r|=o(|V(G_n)|)$. 
Set $T_n = F_n \cap L_n^r$ and apply Proposition \ref{prop}: We get a set $U_n$ such that 
$$|\partial U_n| \leq \alpha |U_n| =  (\alpha+o(1)) |V(G_n)|$$ and
$$|T_n \triangle U_n| \leq \frac{5}{d \kappa^2} |\partial T_n| \leq  \frac{10\varepsilon}{d \kappa^2}|V(G_n)|.$$

Finally, we need to modify $T_n$ in order to get the graph of a $\delta$-almost automorphism. 
For every $x \in V(G)$ define $\pi_1(x)=|\{ (x,y): y \in V(G), (x,y) \in U_n  \}|$  Otherwise I do not see the next inequality in terms of $|\partial U_n|$, and $\pi_2$ similarly. 
Note that $$\sum_{k=0}^{\infty} |E(\pi_1^{-1}([0,k]),\pi_1^{-1}([k+1,\infty)))| \leq |\partial U_n| \leq (\alpha+o(1)) |V(G_n)|.$$
Assume that $\alpha$ is small enough and $n$ is large enough so that $|T_n \triangle U_n| \leq |V(G_n)|/2$, hence
$|\pi_1^{-1}(1)|, |\pi_2^{-1}(1)| \geq |V(G_n)|/2$. Let $h$ denote the Cheeger constant of the expander sequence $(G_n)_{n=1}^{\infty}$.
Since $\pi_1^{-1}([0,k]) \geq |V(G_n)|/2$ for $k \geq 2$ and $|\pi_1^{-1}(0)| \leq |V(G_n)|/2$ we can conclude that 
$$|\pi_1^{-1}(0)| + \sum_{k \geq 2} (k-1) |\pi_1^{-1}(k)| \leq (\alpha+o(1)) |V(G_n)|/h.$$ The same inequality holds for $\pi_2$.

{These two deficiencies are exactly the row and column defects of the bipartite relation $U_n \subset V_n \times V_n$.}
Hence we need to add and remove at most  $(4\alpha+o(1)) |V(G_n)|/h$ vertices in order to get the graph of a bijection. 
The boundary changed by at most $(4\alpha+o(1))d|V(G_n)|/h$. If $\alpha$ is small and $n$ is large enough then
this will be the graph of a $\delta$-almost automorphism $c'_n \colon V(G_n) \to V(G_n)$.

\end{proof}

\section{The group of clusters of almost automorphisms}

Consider the symmetric group ${\rm Sym}(n)$ for $n \in \mathbb N$.
Given a set of permutations $S=\{f_1, \dots ,f_d \} \subset {\rm Sym}(n)$ define the $2d$-regular undirected graph
$G_S$ with possible loops and multiple edges on $V(G_S)=\{1, \dots ,n\}$ and with $E(G_S)$ the disjoint union of the graphs of the maps $f_i$.
The following idea of using expansion originates in the work of Simon Thomas \cite{thomas}.

\begin{lemma}
Consider a set of permutations $S=\{f_1, \dots ,f_d \} \subset {\rm Sym}(n)$ and let $\delta>0$. Moreover, consider two $\delta$-automorphisms of $G_S$, $c,c' \colon V(G_S) \to V(G_S)$.
Then, we have $$d_{\rm H}(c,c') \leq  \frac{2\delta n}{h(G_S)} \quad \mbox{or} \quad n-d_{\rm H}(c,c') \leq  \frac{2\delta n}{h(G_S)}.$$
\end{lemma}

\begin{proof}
Set $A=\{ x \in V(G_S) \mid c(x)=c'(x)\}$. Note that if $x \in A$ and there is an $1 \leq i \leq d$ that $f_i(x) \notin A$ holds then $c$ or $c'$ does not map $(x,f_i(x))$
to the graph of the permutation $f_i$. Hence the number of such elements $x$ in $A$ is at most $2\delta n$. On the other hand, if $|A| \leq n/2$ then we can use the bound on the number of edges between $A$ and $A^c$ to get that this number is at least $h(G_S)|A|$.
Since $|A|=n-d_{\rm H}(c,c')$ this implies the second inequality.
If $|A|>n/2$ we get the first inequality applying the same bound to $A^c$.
\end{proof}

The previous lemma implies that \emph{having small Hamming distance} defines an equivalence relation on $\delta$-almost automorphisms if $\delta$ is small enough.
The next lemma shows that if 
{any $2\delta$-almost automorphism is close to a $\delta$-almost automorphism} 
then we can define a group structure on these equivalence classes in a natural way.

\begin{lemma}
Consider a set of permutations $S=\{f_1, \dots ,f_d \} \subset {\rm Sym}(n)$ and $\delta>0$. Assume that

\begin{enumerate}

\item
{for every $2\delta$-almost automorphism $c$ of $G_S$ there exists a $\delta$-almost automorphism $\alpha(c)$ such that $d_{\rm H}(c,\alpha(c)) \leq n/5$,}

\item
{and for any two $\delta$-almost automorphisms $c,c' \colon V(G_S) \to V(G_S)$, we have $d_{\rm H}(c,c') \notin [n/5, 4n/5]$.
}

\end{enumerate}

Then, we can define a group $\Gamma$ in the following way: The elements of $\Gamma$ are the  equivalence classes of $\delta$-almost automorphisms, where two almost automorphisms are equivalent if their Hamming distance is at most $n/5$. Given two $\delta$-almost automorphisms $c,c' \colon V(G_S) \to V(G_S)$ representing their class define the product of the two classes as the class of $\alpha(c c') \colon V(G_S) \to V(G_S)$. Then, the elements of $\Gamma$ and the binary product are well defined and give a group structure.
\end{lemma}

\begin{proof}
The defined binary relation on $\delta$-almost automorphisms is an equivalence relation: It is clearly reflexive and symmetric. If $a,b,c$
are $\delta$-almost automorphisms and $d_{\rm H}(a,b), d_{\rm H}(b,c) \leq n/5$ then by the triangle inequality $d_{\rm H}(a,c) \leq 2n/5$,
and by $(2)$ we get $d_{\rm H}(a,c) \leq n/5$. This proves transitivity.

Now we show that the product is well defined. Let $a_1, a_2, b_1, b_2$ be $\delta$-almost automorphisms, assume that
$d_{\rm H}(a_1, a_2) \leq n/5$ and $d_{\rm H}(b_1, b_2) \leq n/5$. Then $$d_{\rm H}(\alpha(a_1  b_1),\alpha(a_2  b_2)) \leq 4n/5$$
by $(1)$ and the triangle inequality, hence it can be at most $n/5$ by $(2)$ as required.

The class of the identity will be the identity of $\Gamma$ and the inverse of a class will be the class of the inverses of its elements. Finally, we need to prove associativity. Given three $\delta$-automorphisms $a, b, c$ we suffice to show
$$d_{\rm H}\big(\alpha(a \alpha(b c)),\alpha(\alpha(a b) c) \big) \leq 4n/5.$$
We know that $d_{\rm H} \big((a b) c,\alpha(\alpha(a  b) c) \big) \leq 2n/5$, and similarly \\ 
$d_{\rm H} \big(a (b c), \alpha(a \alpha(b  c)) \big) \leq 2n/5$.
Since $a (b c)=(a b) c$ this gives the required bound. The lemma follows.
\end{proof}

\begin{proof}[Proof of Theorem~\ref{main}]
Consider a sofic approximation of $\Gamma \times \Delta$ and a finite subset $F \subseteq \Delta$.
Assume that the edges of the approximating graph sequence are labeled by elements of $S_{\Gamma} \cup S_{\Delta}$, 
where  $S_{\Gamma} \subset \Gamma, S_{\Delta} \subset \Delta$ are finite subsets. Assume that the sofic approximation of $\Gamma$ 
by the induced subgraphs containing the edges labeled by elements of $S_{\Gamma}$ is an essentially expander sequence.

Let $\delta >0$ small and $n$ large enough such that the clusters of $\delta$-almost automorphisms form a group, where two 
$\delta$-almost automorphisms are in the same cluster if their Hamming distance is at most a fifth of the size of the vertex set.
Moreover, let $n$ be large enough (and hence the "error of the sofic approximation" small enough) such that the elements of $FF$ 
are all in different clusters, and for any $x, y \in F$ the cluster of $xy$ is the product of the clusters of $x$ and $y$ (in the group of clusters). 
We know that the group of clusters is finite. Since this holds for any finite $F \subseteq \Delta$ the group $\Delta$ is LEF.
\end{proof}

\begin{remark} It is a very interesting problem to classify or understand almost subgroups of ${\rm Sym}(n)$, such as the almost centralizers of expander sofic approximations of Kazhdan groups arising from our results.
In this context, it seems to be an open problem\footnote{{After circulation of the first draft of this paper, this question was answered affirmatively by Becker and Chapman \cite{BeckerChapman}. They proved uniform flexible
stability in permutations for all amenable groups.}} to decide whether every uniform almost homomorphism from a finite group to ${\rm Sym}(n)$ is uniformly close to a homomorphism to say ${\rm Sym}((1+\varepsilon)n)$. However, the analogous question for ${\rm U}(n)$ (equipped with the normalized Hilbert-Schmidt distance) has been answered in \cites{MR3733361,MR3867328}, which already gives some information since ${\rm Sym}(n) \subset {\rm U}(n)$ in a way compatible with the metrics.
\end{remark}

\section{The proof of Theorem \ref{product}} \label{sec:product}

We will us the following basic lemma on LEF groups.

\begin{lemma}\label{LEF}
If $\Delta$ is a finitely generated group and
$\Delta_0 \leq \Delta$ is a finite index LEF subgroup, then
$\Delta$ is LEF.
\end{lemma}

\begin{proof}
{Let
$
M= \bigcap_{\delta\in \Delta} \delta \Delta_0 \delta^{-1}.
$
Since $\Delta_0$ has finite index in $\Delta$, the subgroup $M$ is normal
and of finite index in $\Delta$. Moreover, $M\leq \Delta_0$, hence $M$ is
LEF.

By the Kaloujnine--Krasner embedding theorem \cite{KK51}, the group
$\Delta$ embeds into the wreath product
$
M \wr (\Delta/M)
$
with finite quotient $\Delta/M$. By Gordon--Vershik \cite{GV97}, the wreath product of an LEF group with a finite group is LEF.
Thus, $M\wr(\Delta/M)$ is LEF. It
follows that $\Delta$ is LEF.}
\end{proof}

We can assume that the finite set of generators of $\Gamma \times \Delta$ is the union of the set of generators $S_{\Gamma}$ of $\Gamma$ and the set of generators $S_{\Delta}$ of $\Delta$.
Consider a sofic approximation of $\Gamma \times \Delta$ by a sequence of expander graphs. We know that the spanning subgraphs corresponding to the edges with labels in $\Gamma$ are essentially a disjoint union of expander graphs. Hence, by changing a small number of these edges we get a sofic approximation of $\Gamma \times \Delta$, whose restriction to the $\Gamma$-labeled edges is a disjoint union of expander graphs. We will refer to these expanders as $\Gamma$-components. 

Our next goal is to change the sofic approximation in order to get a new one such that every $\Gamma$-component has the same size, moreover, the equivalence relation of being in the same $\Gamma$-component is compatible with the action of the elements in $\Sigma_{\Delta}$. We call an edge $(x,y)$ labeled by $\delta \in S_{\Delta}$ {\it good} if for more than half of the vertices in the $\Gamma$-component of $x$ the endvertex of the unique edge labeled by $\delta$ and starting at that vertex is in the $\Gamma$-component of $y$, and $(x,y)$ is called {\it bad} otherwise. We will change the sofic approximation such that every edge with label in $S_{\Delta}$ is good. Note that this implies by connectivity that all the $\Gamma$-components have equal size. We need the following claims.

\begin{claim}
Consider a sofic approximation of $\Gamma \times \Delta$ such that all the
$\Gamma$-components are expanders. Then the proportion of the bad edges goes to zero.  
\end{claim}
\begin{proof}
Given the $\Gamma$-components $C, D$ and the element $\delta \in S_{\Delta}$ let $C' \subset C$ denote the set of vertices, such that the edge labeled by $\delta$ starting at at these vertices ends in $D$. If these edges are bad then $|C'| \leq |C|/2$, hence $|\partial_{\Gamma} C'| \geq h |C|$, where $\partial_{\Gamma}$ denotes the edge boundary in the spanning subgraph with the set of edges with labels in $\Gamma$. Note that for every bad edge $(x,y)$ labeled by $\delta$, where $x \in C', y \in D$, and edge in $\partial_{\Gamma} C'$ labeled by $\gamma \in S_{\Gamma}$ and starting at $x$, the unique walk of length four starting at $y$ and labeled by the commutator of $\delta$ and $\gamma$ does not end at $y$, since it ends in a component different from $D$. Thus, every edge on the boundary $\partial_{\Gamma} C'$ contributes to the error of the sofic approximation. Since this holds for every $\delta \in S_{\Delta}$ and $C,D$ we get by summation that this error is at least the Cheeger constant times the number of the bad edges. The claim follows by the definition of sofic approximation.
\end{proof}

Given a vertex $x$ let $C_x$ denote its $\Gamma$-component.

\begin{claim}
For every $\varepsilon>0$ the proportion of the good edges $(x,y)$ such that $|C_x|>(1+\varepsilon)|C_y|$ goes to zero. 
\end{claim}
\begin{proof} 
Note that if $(x,y)$ is a good edge and $|C_x|>(1+\varepsilon)|C_y|$, then there are at least $|C_x|-|C_y|>\varepsilon |C_x|$ bad edges with the same label starting in $C_x$, namely the edges whose other endvertex is not in $C_y$. Summing this up for every $\Gamma$-component yields the claim, since the proportion of the bad edges goes to zero. 
\end{proof}

\begin{claim}
There is a sofic approximation of $\Gamma \times \Delta$ such that all the $\Gamma$-components are expanders and there are no bad edges.
\end{claim}
\begin{proof}
Consider a sofic approximation such that the $\Gamma$-components are expanders. First, we show that the size of most $\Gamma$-components is close to be equal.
Let $m$ denote the median size of the $\Gamma$-components, that is, the positive integer such that at least half of the vertices are contained by a $\Gamma$-component of size at least $m$, and at least half of the vertices are contained by a $\Gamma$-component of size at most $m$. We prove for any $\varepsilon>0$ that the proportion of vertices whose $\Gamma$-component has more than $(1+\varepsilon)m$ or less than $(1-\varepsilon)m$ vertices goes to zero.

Let $0<\varepsilon'<\frac{1}{2}$. We might assume that the proportion of edges $(x,y)$ such that $|C_x|>(1+\varepsilon')|C_y|$ is at most $\frac{h\varepsilon'}{2}$. Then for every $m'<m$ if $\{x: |C_x| \leq m'\}$ contains at least an $\varepsilon'$ proportion but at most half of the vertices then $|\{x: |C_x| \leq (1+\varepsilon')m'\}|>(1+h/2d)|\{x: |C_x| \leq m'\}|$, since 
$|\partial_{\Delta} \{x: |C_x| \leq m'\}| \geq h |\{x: |C_x| \leq m'\}|$, and at least half of the edges goes to vertices in $\Gamma$-components of size at most $(1+\varepsilon')m'$. Thus, by iteration we get that if we choose $m'$ such that $\{x: |C_x| \leq m'\}$ contains an $\varepsilon$ proportion of the vertices and $k$ such that $(1+h/2d)^k>\frac{1}{2\varepsilon}$ then $\{x: |C_x| \leq (1+\varepsilon')^k m'\}$ contains at least half of the vertices, hence  $m'\geq (1+\varepsilon')^{-k} m$.
Since this holds for every $\varepsilon'>0$, we obtain that for every $\varepsilon>0$ asymptotically all but an $\varepsilon$ proportion of the vertices is contained by $\Gamma$-components of size at least $(1-\varepsilon)m$. The proof of the upper bound is similar.  

Moreover, we know that the proportion of bad edges go to zero. We can remove a small number of vertices and possibly change a small number of edges in order to end up at every $\Gamma$-component having the same size. In order to add the missing $\Delta$-edges first we consider the graph whose vertices are the components, and two are adjacent (with a labeled edge) if there are good edges connecting them: we extend this to an $S_{\Delta}$ labeled graph. Finally, we add edges between two $\Gamma$-components if they are connected in the graph of components, the additional edges get the same label. This completes the proof of the claim.
\end{proof}

We distinguish two cases. First, assume that the number of $\Gamma$-components is constant (for an infinite subsequence), denote it by $k$. Let $C$ denote one of the $\Gamma$-components. Consider the action of the free group $F_{S_{\Delta}}$ on the $\Gamma$-components: this is well-defined, since all the edges are good.
Note that the set of elements in $F_{S_{\Delta}}$ fixing $C$ (as a set) is a subgroup of index $k$, whose generators represent elements of $\Delta$ inducing a subgroup $\Delta_0 \leq \Delta$ of index $k$. We can assume that we get the same subgroup $\Delta_0 \leq \Delta$ for every graph in the sofic approximation. 
Altogether, we get a sofic approximation of $\Gamma \times \Delta_0$ on $C$, that is an expander with respect to $S_{\Gamma}$. Thus, by Theorem \ref{main} the group $\Delta_0$ is LEF, and by Lemma \ref{LEF} the group $\Delta$ is LEF too.

Second, assume that the number of $\Gamma$-components is unbounded; after
passing to a subsequence we may assume that it tends to infinity.
Let $q_n$ be the number of $\Gamma$-components of $G_n$, and let $m_n$ be their
common size. Thus,
$
        |V(G_n)| = q_n m_n$ and $
        q_n\to \infty.$
Consider the graph $Q_n$ whose vertices are the $\Gamma$-components of $G_n$.
For $\delta\in S_\Delta$ we put a $\delta$-edge from $C$ to $D$ if the
$\delta$-edges starting in $C$ end in $D$. This is well-defined since there are no
bad edges.

{
Given a vertex $x$ and a positive integer $k$ let $p_{2k}(x)$ denote the probability that a random walk starting at $x$ using edges in $S_{\Delta}$ returns to $x$. For the Cayley graph of $\Delta$ the Kesten criteria for amenability implies for every $\lambda<1$ that $p_{2k}(x)^{1/2k}>\lambda$ if $k$ is large enough. This also holds for every vertex $x \in V(G_n)$ if $B(x,k)$ is isomorphic to the $k$-ball in the Cayley graph of $\Delta$. There are such vertices by sofic approximation (in fact, their proportion goes to one). 

Note that for every $x \in V(G_n)$ the probability of return for its $\Gamma$-component $C_x$, as a vertex of $Q_n$, can not be smaller than the probability of return for $x$: if a walk corresponding to a word in $\Delta$ returns to $x$ then the walk corresponding to the same word starting at $C_x$ should also return. Thus, for every $\lambda>0$ if $k$ is large enough and also $n$ depending on $k$ and the sofic approximation is large enough then we can find a vertex $y \in V(Q_n)$ such that $p_{2k}(y)^{1/2k}>\lambda$ holds. 

On the other hand, for an expander sequence there exists $\lambda<1$ such that $p_{2k}(x)^{1/2k}<\lambda$ holds for every vertex $x$ if the number of vertices is large enough, that is, at least an exponentially large function of $k$. 
Hence the sequence $(Q_n)_n$ can not be an expander sequence. This completes the proof of the theorem.}



\section*{Acknowledgments}

The work of the authors has received funding from the European Research Council (ERC) under the European Union’s Horizon 2020 research and innovation programme (Consolidator Grant No. 681207, Consolidator Grant No. 617747 and Advanced Grant No. 741420) by the Hungarian Academy of Sciences Momentum Grant no. 2022-58.


\end{document}